\documentclass[12pt,leqno]{article}

\usepackage{graphicx, fullpage}
\usepackage{amsmath,amssymb,amsthm,amscd, bm}
\usepackage{fancyhdr}
\usepackage[mathscr]{eucal}
\usepackage{amsfonts}
\usepackage[T1]{fontenc}
\usepackage{xspace}
\usepackage{esint}
\begin{document}
\parskip=6pt

\theoremstyle{plain}
\newtheorem{prop}{Proposition}
\newtheorem{lem}[prop]{Lemma}
\newtheorem{thm}[prop]{Theorem}
\newtheorem{cor}[prop]{Corollary}
\newtheorem{defn}[prop]{Definition}
\theoremstyle{definition}
\newtheorem{example}[prop]{Example}
\theoremstyle{remark}
\newtheorem{remark}[prop]{Remark}
\numberwithin{prop}{section}
\numberwithin{equation}{section}

\newenvironment{rcases}
  {\left.\begin{aligned}}
  {\end{aligned}\right\rbrace}


\def\cal{\mathcal}
\newcommand{\cF}{\cal F}
\newcommand{\cG}{\cal G}
\newcommand{\cA}{\cal A}
\newcommand{\cB}{\cal B}
\newcommand{\cC}{\cal C}
\newcommand{\cO}{{\cal O}}
\newcommand{\cE}{{\cal E}}
\newcommand{\cU}{{\cal U}}
\newcommand{\cM}{{\cal M}}
\newcommand{\cD}{{\cal D}}
\newcommand{\cK}{{\cal K}}
\newcommand{\cZ}{{\cal Z}}

\newcommand{\fQ}{\frak{Q}}

\newcommand{\bC}{\mathbb C}
\newcommand{\bP}{\mathbb P}
\newcommand{\bN}{\mathbb N}
\newcommand{\bA}{\mathbb A}
\newcommand{\bR}{\mathbb R}
\newcommand{\oP}{\overline P}
\newcommand{\oQ}{\overline Q}
\newcommand{\oR}{\overline R}
\newcommand{\oS}{\overline S}
\newcommand{\oc}{\overline c}
\newcommand{\bp}{\mathbb p}
\newcommand{\oD}{\overline D}
\newcommand{\oE}{\overline E}
\newcommand{\oC}{\overline C}
\newcommand{\of}{\overline f}
\newcommand{\ou}{\overline u}
\newcommand{\oU}{\overline U}
\newcommand{\oV}{\overline V}
\newcommand{\ow}{\overline w}
\newcommand{\oy}{\overline y}
\newcommand{\oz}{\overline z}

\newcommand{\hg}{\hat G}
\newcommand{\hM}{\hat M}

\newcommand{\tpr}{\widetilde {\text{pr}}}
\newcommand{\tB}{\widetilde B}
\newcommand{\tx}{\widetilde x}
\newcommand{\ty}{\widetilde y}
\newcommand{\txi}{\widetilde \xi}
\newcommand{\teta}{\widetilde \eta}
\newcommand{\tna}{\widetilde \nabla}
\newcommand{\tth}{\widetilde \theta}

\newcommand{\diml}{\text{dim}}
\newcommand{\var}{\varepsilon}
\newcommand{\End}{\text{End }}
\newcommand{\loc}{\text{loc}}
\newcommand{\Symp}{\text{Symp}}
\newcommand{\Sympo}{\text{Symp}(\omega)}
\newcommand{\lam}{\lambda}
\newcommand{\Hom}{\text{Hom}}
\newcommand{\Ham}{{\rm{Ham}}}
\newcommand{\ham}{\text{ham}}
\newcommand{\Ker}{\text{Ker}}
\newcommand{\dist}{\text{dist}}
\newcommand{\psl}{\rm{PSL}}
\newcommand{\rk}{\roman{rk }}
\newcommand{\id}{\text{id}}
\newcommand{\Det}{\text{Det}\,}
\renewcommand\qed{ }
\begin{titlepage}
\title{\bf On the adjoint action of the group of symplectic diffeomorphisms}
\author{L\'aszl\'o Lempert \thanks{Research partially  supported by NSF grant DMS 1764167.
\newline 2020 Mathematics subject classification 53D05, 58D19}\\ Department of  Mathematics\\
Purdue University\\West Lafayette, IN
47907-2067, USA}
\thispagestyle{empty}
\end{titlepage}
\date{}
\maketitle
\abstract
We study the action of Hamiltonian diffeomorphisms of a compact symplectic manifold ($X,\omega$) on $C^\infty(X)$ and 
on functions $C^\infty(X)\to \bR$. We describe various properties of invariant convex functions on $C^\infty(X)$.
Among other things we show that continuous convex functions $C^\infty(X)\to \bR$ that are invariant under the action 
are automatically invariant under so called strict rearrangements and they are continuous in the sup norm topology of $C^\infty(X)$; but this is not generally true if the convexity condition is dropped.
\endabstract

\section{Introduction}

Consider a connected, compact, symplectic manifold $(X,\omega)$, without boundary, of dimension $2n$. According to Omori 
\cite{O}, symplectic self--diffeomorphisms of $X$ form a Fr\'echet--Lie group $\Symp(\omega)$, with Lie algebra the space $\frak{v}(\omega)$ of smooth vector fields on $X$ that are locally Hamiltonian. In this paper we will be interested in the action of $\Sympo$, by pull back, on the Fr\'echet space $C^\infty(X)$ of smooth real functions
\begin{equation}  
\Sympo \times C^\infty(X)\ni(g,\xi)\mapsto\xi\circ g^{-1}\in C^\infty(X),
\end{equation}
and on functions on $C^\infty(X)$ that (1.1) leaves invariant. This action is no adjoint action, but it is close to one. The adjoint action Ad$_g$ of $g\in \Sympo$ is, rather, push forward by $g^{-1}$ of vector fields in $\frak{v}(\omega)$. The
subspace $\ham(\omega)\subset\frak{v}(\omega)$ of globally Hamiltonian vector fields, those that are symplectic gradients  sgrad$\,\xi$ of some $\xi\in C^\infty(X)$, is invariant under Ad$_g$, and (1.1) induces via the projection $\xi\mapsto\text{sgrad}\,\xi$  the restriction of the adjoint action to $\ham(\omega)$.

Other diffeomorphism groups of $X$ also act on $C^\infty(X)$ by pull back. Our focus will be on the subgroup $\Ham(\omega)\subset\Sympo$ of Hamiltonian diffeomorphisms. Hamiltonian diffeomorphisms are the time 1 maps 
 of time dependent Hamiltonian vector fields sgrad$\, \xi_t$, $\xi_t\in C^\infty(X)$. Continuous norms---and also
 seminorms---on the Fr\'echet space $C^\infty(X)$, invariant under $\Ham(\omega)$, are of potential interest in symplectic geometry because they give rise to bi--invariant metrics on
$\Ham(\omega)$, and have been investigated in the past. An obvious norm is $\|\xi\|_\infty =\max_X|\xi|$. That it gives rise to a genuine metric on $\Ham(\omega)$ was proved first by Hofer in $\bR^{2n}$, and in general by Lalonde and McDuff; see 
also Polterovich's book, [Ho, LM, P]. Work by Ostrover--Wagner, Han, and Buhovsky--Ostrover [BO, Ha, OW] gave the following.  Let 
$(X,\mu)$ and $(Y,\nu)$ be measure spaces. We say that measurable functions $\xi:X\to\bR$, $\eta:Y\to\bR$
are equidistributed, or strict rearrangements of each other, if
\[\mu(\xi^{-1}B)=\nu(\eta^{-1}B) \quad\text{ for all Borel sets } B\subset \bR.
\]
When  $\mu(X),\nu(Y)<\infty$, 
this is equivalent to 
$\mu\{x\in X: \xi(x)> t\}=\nu\{y\in Y:\eta(y)>t\}$  for all  $t\in\bR$. We have to use the qualifier `strict', since the notion of 
rearrangement  in harmonic analysis and Banach space theory typically refers to the relation
$\mu\{x\in X: |\xi(x)|> t\}=\nu\{y\in Y:|\eta(y)|>t\}$.
Back to our symplectic manifold $(X,\omega)$, we write $\mu$ for the measure on $X$ defined by $\omega^n$; 
the action (1.1) clearly sends functions on $(X,\mu)$
to their strict rearrangements.

\begin{thm}[{\rm{[BO, H, OW]}}] 
If $\|\ \  \|$ is a $\Ham(\omega)$ invariant continuous seminorm on the Fr\'echet space $C^\infty(X)$, then $\|\xi\|=\|\eta\|$ whenever $\xi,\eta\in C^\infty(X)$ are equidistributed. 
These seminorms satisfy $\|\ \  \|\le c\|\ \  \|_\infty$ with some $c\in(0, \infty)$.
Unless $\|\ \ \|$ and $\|\ \  \|_\infty$ are equivalent, the pseudodistance on $\Ham(\omega)$ induced by $\|\ \ \|$ is 
identically 0.
\end{thm}

One of our goals in this paper is to offer a simpler proof to the first two statements, in fact in a slightly greater generality: 

\begin{thm}   
Suppose $p:C^\infty(X)\to \bR$ is a continuous, convex function that is invariant under the action of $\Ham(\omega)$. Then $p$ is continuous in the topology of $C^\infty(X)$ induced by $\|\ \  \|_\infty$, and is invariant under strict
rearrangements: $p(\xi)=p(\eta)$ whenever $\xi, \eta$ are equidistributed. 
\end{thm}

The point is not the modest gain in generality, which can easily be achieved once Theorem 1.1 is known (for example 
along the lines of  the proof
of Theorem 4.1 below). Rather, it is the simplification of the proof. This is how the two proofs compare.
\cite{OW} first proved that any $\Ham(\omega)$ invariant seminorm $\|\ \ \|\le c\|\ \ \|_\infty$
is invariant under volume preserving diffeomorphisms. Han in \cite{Ha} 
subsequently strengthened this to invariance under strict rearrangements. 
All this is obtained as a consequence of a lemma of Katok \cite[Section 3]{K}. The final step is in \cite{BO}, that takes an arbitrary continuous $\Ham(\omega)$ invariant seminorm $\|\ \ \|$ on $C^\infty(X)$, and proves by an involved argument that $\|\ \ \|\le c\|\ \ \|_\infty$.

We obtain the simplification by restructuring the proof. First we prove that $p$ in Theorem 1.2 is a limit point of the set of $\Ham(\omega)$ invariant functions $q$ that are continuous in the $L^1$ topology on $C^\infty(X)$. This depends on studying linear forms on $C^\infty(X)$, i.e., distributions, and regularizing them using the action of $\Ham(\omega)$. Katok's lemma now gives that the functions $q$ are invariant under strict rearrangements, whence so must be their limit point $p$. Another application of Katok's lemma, combined with  real analysis type arguments then gives the continuity of $p$ with respect to $\|\ \ \|_\infty$. 

Continuity of $p$ with respect to $\|\ \ \|_\infty$ in Theorem 1.2 is essentially an upper estimate of $p$.  We will also prove
a lower estimate:

\begin{thm} 
Let $p\colon C^\infty(X)\to\bR$ be $\Ham(\omega)$ invariant,
convex, and continuous.  Then 
either

(i) $p(\xi)=p_1(\int_X\xi\omega^n)$, where $p_1\colon\bR\to\bR$ is convex; 
or
  
(ii) there are $a\in\bR$, $b\in(0,\infty)$ such that
\[
p(\xi)\ge a+b\int_X|\xi|\omega^n\quad\text{ if }\quad\begin{cases} \int_X\xi\omega^n=0,{\text{ or}}\\ 
 \int_X\xi\omega^n\ge0\ \text{ and }\ \lim_{\bR\ni \lambda\to\infty}p(\lambda)=\infty,\quad{\text{ or}} \\
\int_X\xi\omega^n\le 0\ \text{ and }\ \lim_{\bR\ni \lambda\to-\infty}p(\lambda)=\infty.\end{cases}
\]
If $p$ is positively homogeneous ($p(c\xi)=cp(\xi)$ for positive constants $c$), then $a=0$. 
\end{thm}

In particular, if $p$ is a norm, then it dominates $L^1$ norm, something that [OW] also found (cf. Proposition 6.1 there and
its proof).   

Above we have insisted on the difference between rearrangements and strict rearrangements. Nevertheless, Theorem 
1.3 implies that in our setting
the difference between the two is minimal. The notion of rearrangement invariant Banach spaces in the next theorem
is defined in \cite{BS}, 
see also section 6.

\begin{thm}
Given a $\Ham(\omega)$ invariant continuous norm $p$ on $C^\infty(X)$, there is a rearrangement invariant Banach 
function space on $X$ whose norm, restricted to $C^\infty(X)$, is equivalent to $p$.
\end{thm}

A natural question is whether Theorem 1.2 holds for all continuous $\Ham(\omega)$ invariant functions $p$, independently
of convexity. It does not:

\begin{thm}  
If $\dim X \ge 4$, there is a smooth $\Ham(\omega)$ invariant function $p:C^\infty(X)\to\bR$ that is not invariant under volume preserving diffeomorphisms $X\to X$.
\end{thm}

The last statement of Theorem 1.1 suggests that, after all, the only invariant norm on $C^\infty(X)$ that is of interest for symplectic geometry, is Hofer's norm $\|\ \ \|_\infty$. However, all invariant norms are of interest for K\"ahler geometry. The groups $\Sympo$ and $\Ham(\omega)$ can be regarded as symmetric spaces. When $(X,\omega)$ is K\"ahler, 
Donaldson, Mabuchi, and Semmes proposed that the infinite dimensional manifold $\cal{H}_\omega$ of relative K\"ahler potentials, endowed with a natural connection on its tangent bundle, should be viewed as the dual symmetric space, at least in a formal sense;
 see [Do, M, S1, S2]. $\Ham(\omega)$ invariant norms on $C^\infty(X)$ induce Finsler metrics on $\cal{H}_\omega$ that are invariant under parallel transport, and, perhaps surprisingly, all these Finsler metrics
induce genuine metrics on $\cal{H}_\omega$. Mabuchi was the first to study such a metric, associated with $L^2$-norm $\|\xi\|=(\int_X|\xi|^2\omega^n)^{1/2}$; more recently, Darvas in \cite{Da} introduced various Orlicz norms on $C^\infty(X)$ and the induced metrics on $\cal{H}_\omega$. Generalizing Darvas's norms and metrics, in \cite{L} we study general 
 $\Ham(\omega)$ invariant Lagrangians and the associated action on $\cal H_\omega$, and most  results
here are motivated by the needs of that paper.
\section{Reduction to linear forms}  

In this section $(X,\omega)$ can be any $2n$ dimensional symplectic manifold, not necessarily compact. The space of compactly supported smooth functions on $X$ will be denoted $\cal{D}(X)$, with its usual locally convex inductive limit topology. Its dual is $\cal{D}'(X)$, the space of distributions. The group $\Ham_0(\omega)$, time 1 maps of compactly supported
Hamiltonian flows, acts on $\cal{D}(X)$ by pull back and on $\cal{D}'(X)$ by push forward. We denote the pairing between $\cal{D}'(X)$ and $\cal{D}(X)$ by $\langle\ ,\ \rangle$. The locally convex topology of $\cal{D}'(X)$ is generated by the seminorms $\| f \|'_\xi=|\langle f,\xi\rangle |$ with $\xi\in \cal{D}(X)$. 
Integration against any smooth $2n$--form defines a distribution. Such distributions will be called smooth. If $h\in \cal{D}'(X)$, 
we denote by conv$(h)$ the closed convex hull of the $\Ham_0(\omega)$ orbit of $h$.

The main result of this section is
\begin{lem}  
Suppose $p: \cal{D}(X)\to \bR$ is a $\Ham_0(\omega)$ invariant, continuous, convex function. There is a  family $\cal{A}\subset \bR\times C^\infty(X)$ such that
\begin{equation}  
p(\xi)=\sup\Big\{a+\int_X f\xi\omega^n : (a,f)\in \cal{A}\Big\}, \quad\text{ for all }\ \xi\in \cal{D}(X).
\end{equation}
If $p$ is positively homogeneous as well $(p(c\xi)=c p(\xi) \text{ for } 0<c<\infty)$, then $\cal{A}$ can be chosen in $\{0\}\times C^\infty(X)$.
\end{lem}

For the proof we need certain regularization maps $\cal{D}'(X)\to \cal{D}'(X)$. Let $U\subset\subset X $ be open, and assume that 
on a neighborhood of $\oU$ there are local coordinates $x_\nu$ in which $\omega$ takes the form 
$\sum^n_1 dx_\nu\wedge dx_{n+\nu}$. Let $C\subset X\setminus\oU$  be compact. Fix $\varphi_\nu\in \cal{D}(X)$, 
$\nu=1,\dots, 2n$,  vanishing on a neighborhood of $C$, such that $\varphi_\nu=x_\nu$ in a neighborhood of $\oU$. Let 
$g^\tau_\nu, \tau\in\bR$, denote the Hamiltonian flow of 
$\varphi_\nu$ for $\nu\le n$ and of $-\varphi_\nu$ for $\nu>n$; i.e., the flow of the vector fields $\pm\text{sgrad} \,\varphi_\nu$. 
If $t=(t_1,\ldots,t_{2n})\in\bR^{2n}$, put 
\[
g^t=g_1^{t_1}\circ g^{t_2}_2\circ \dots \circ g^{t_{2n}}_{2n}.
\]
Near $C$ we have $g^t=\id$; on $U$, for small $t$, $g^t(x)=x-t$. Let furthermore $\chi\in \cal{D}(\bR^{2n})$ be nonnegative, 
$\int_{\bR^{2n}}\chi(t) dt_1\dots dt_{2n}=1$. For $\lambda\in(0,\infty)$
define operators $R_\lambda:\cal{D}'(X)\to \cal{D}'(X)$ by
\[
R_\lambda h=\lambda^{2n}\int_{\bR^{2n}}\chi(\lambda t)(g^t_* h) dt_1\dots dt_{2n}\in \text{conv}(h),\quad h\in \cal{D}'(X).
\]

Standard properties of convolutions imply
\begin{lem}  
$\lim_{\lambda\to\infty}R_\lambda h=h$ for $h\in\cal{D}'(X)$.
If the support of $\chi$ is sufficiently close to 0, then $R_\lambda h\in\text{conv}(h)$ is smooth on $U$ and $R_\lambda h=h$ on a neighborhood of $C$. Furthermore, if 
$V\subset\subset W\subset X$ are open, and $h$ is smooth on $W$, then $R_\lambda h$ is smooth on $V$ for sufficiently large $\lambda$.
\end{lem}

\begin{lem}  
For any $h\in \cal{D}'(X)$, smooth distributions are dense in $\text{conv}(h)$.
\end{lem}
\begin{proof}
It will suffice to prove that given a finite $\Xi\subset \cal{D}(X)$ and $\var>0$, there is a smooth $h'\in\text{conv}(h)$ such that 
$|\langle h'-h,\xi\rangle|\le \var$ for all $\xi\in \Xi$. To show this latter, for each $z\in X$ construct an open neighborhood 
$V(z)\subset\subset X$ so that in a neighborhood of $\overline{V(z)}$ we can write 
$\omega=\sum dx_\nu\wedge dx_{n+\nu}$ in suitable 
local coordinates. Select a locally finite cover $V(z_1), V(z_2),\dots$ of $X$. Thus the $V(z_j)$ form a finite 
or infinite cover 
depending on whether $X$ is compact or not. For each $j$ we can find $U_j\supset\supset V(z_j)$ such that 
$\{U_j\}_j$ is still locally finite, and $\omega=\sum dx_\nu\wedge dx_{n+\nu}$ is still valid in some neighborhood of 
$\overline{U_j}$. 
Fix furthermore open sets $V^i_j$, $i\in\bN$, such that
\[
U_j=V^1_j\supset\supset V^2_j\supset\supset\dots\supset V(z_j),
\]
and compact sets $C_j\subset X\setminus \bigcup_{k>j}U_k$, $C_0=\emptyset$, such that $C_{j-1}\subset\text{int} \,C_j$ and 
$\bigcup_jC_j=X$. We let $h_0=h$ and construct $h_j\in \text{conv}(h)$ so that for $j\ge 1$
\begin{itemize}
\item[] \ \  $|\langle h_j-h,\xi\rangle | <\var\quad \text{if }\xi\in\Xi$;
\item[] \ \ $ h_j |V_1^j\cup\dots\cup V^j_j  \quad \text{is smooth }$;
 \item[]\ \  $h_j=h_{j-1} \quad \text{on int\,} C_{j-1}.$
\end{itemize}

Assuming we already have $h_{j-1}$, we apply Lemma 2.2 with $U=U_j$, $C=C_j$, $V=V_1^j\cup\dots\cup V^j_{j-1}$, and $W=V^{j-1}_1 \cup\dots \cup V^{j-1}_{j-1}$. If $\lambda$ is sufficiently large, then $h_j=R_\lambda h_{j-1}$ will do as the next function. Note that $h_j$ is smooth over $V\cup U_j\supset V^j_1\cup\dots\cup V^j_{j-1}\cup V^j_j$. 

Thus $h_j=h_{j+1}=\dots$ on int$\,C_j$ and $h_j |V(z_1)\cup\dots\cup V(z_j)$ is smooth. If $X$ is compact, we take $h'$ to be the last $h_j$; otherwise we take $h'=\lim\limits_{j\to\infty} h_j$.
\end{proof}

\begin{proof}[Proof of Lemma 2.1]
By an affine function we mean a function $\cD(X)\to\bR$ of the form $\text{const} +\text{linear}$. Clearly, if an affine 
function is bounded above on a symmetric neighborhood of $0\in \cD(X)$, it is bounded below as well, hence continuous.

Let $\cal{B}$ denote a collection of affine functions $\beta: \cD(X)\to\bR$ such that $\beta\le p$. Thus $\beta\in \cB$ can be written 
\begin{equation}  
\beta(\xi)=a+\langle h,\xi\rangle, \quad\text{ with } a\in\bR, \quad h\in \cal{D}'(X).
\end{equation}

The Banach--Hahn separation theorem gives that $p=\sup_{\beta\in \cB}\beta$ with a suitable choice of $\cB$. If $p$ is positively homogeneous, another version of the Banach--Hahn theorem, see e.g. \cite[p.317-319]{Sc}, gives that $\cB$ can be taken to consists of linear forms, i.e. all $a$ will be 0.

By the invariance of $p$, if $\beta$ in (2.2) is in $\cB$, then for any $g\in\Ham(\omega)$
\[
(g_*\beta)(\xi)=a+\langle g_* h, \xi\rangle=a +\langle h, g^*\xi\rangle\le p(\xi).
\]
This means that all $g_*\beta$ can be adjoined to $\cB$, and in fact we can arrange that all $a+\langle h',\xi\rangle$ are
in $\cB$ for any $h\in\cB$ and $h'\in\text{conv}(h)$. Therefore if we take all 
$\beta\in \cB$ of form (2.2) with smooth $h$ and write $h$ as $f\omega^n$, the family $\cal{A}$ of pairs $(a,f)$ 
thus obtained will do according to Lemma 2.3.
\end{proof}

\section{Proof of the second part of Theorem 1.2}  

The second part was:

\begin{thm}  
Let $(X,\omega)$ be a connected, compact, symplectic manifold. Any  continuous, convex, and $\Ham(\omega)$ invariant function
 $p: C^\infty(X)\to\bR$ is strict rearrangement invariant: $p(\xi)=p(\eta)$ if $\xi,\eta$ are equidistributed.
\end{thm}

As before, $\mu$ denotes the Borel measure on $X$ that the form $\omega^n$ determines. In our integrals below we will often omit $d\mu$ and write $\int_Ef$ for $\int_Ef\, d\mu$; and when $E=X$, we will even omit $X$ and write $\int f$ for $\int_Xf\,d\mu$. In the same spirit, we write $L^q(X)$ for $L^q(X,\mu)$.

We need the following result, an equivalent of Katok's Basic Lemma, valid for noncompact (but connected) $X$ as well:

\begin{lem}  
If $\xi,\eta\in L^1(X)$ are equidistributed, then there is a sequence of $g_k\in\Ham_0(\omega)$ such that
\[\lim\limits_{k\to\infty} \int_X |\xi-\eta\circ g_k| d\mu=0.\]
\end{lem}
\begin{proof}
(Essentially as in \cite{OW}, \cite[Proposition 1.12]{Ha}.) Given $\var>0$, we will find $g\in\Ham_0(\omega)$ such that
$\int|\xi-\eta\circ g|<5\var$. Assume first $\mu(X)<\infty$. 

The measure $|\xi|d\mu$ is absolutely  continuous with respect to $d\mu$, hence there is a $\delta>0$ such that 
$\int_E |\xi| <\var$ if  $\mu(E)<\delta$.
Construct disjoint intervals $J_1,\dots,J_N\subset \bR$ of length $<\var/\mu (X)$ so that 
$\mu(X\setminus\bigcup_i\xi^{-1} J_i) < \delta/2$, and choose compact sets $K_i\subset \xi^{-1}J_i$ 
so that also
\begin{equation} 
\mu(X\setminus \bigcup_i K_i) < \delta/2.
\end{equation}
By equidistribution $\mu(\eta^{-1}J_i)=\mu(\xi^{-1} J_i)$, hence there are compact $L_i\subset \eta^{-1}J_i$
such that $\mu(L_i)=\mu(K_i)$. The $K_i$ are disjoint among themselves and so are the $L_i$. In this situation Katok's Basic Lemma \cite[Section 3]{K} provides a $g\in \Ham_0(\omega)$ such that
\begin{equation}   
\mu(K_i\setminus g^{-1} L_i) < \delta/2N,\quad i=1,\dots,N.
\end{equation}

If $x\in K_i\cap g^{-1}L_i$ then $\xi(x)$, $\eta(gx)\in J_i$ and so $|\xi(x)-\eta(gx)| < \var/\mu (X)$. Conversely, 
$|\xi(x)-\eta(gx)| \ge \var/\mu(X)$ can happen only if
\[x\in E, \quad\text{ where } E=\big(X\setminus \bigcup_i K_i\big)\cup \bigcup_i\big(K_i\setminus g^{-1} L_i\big).\]
By (3.1), (3.2) $\mu(E)<\delta$, whence $\mu(gE)<\delta$ and 
\[
\int |\xi-\eta\circ g|= \int_{X\setminus E} |\xi-\eta\circ g| + \int_E |\xi-\eta\circ g| < \var + \int_E |\xi| + \int_{gE} |\eta| < 3\var.
\]

This takes care of $X$ of finite measure. In general, choose an $a>0$ so that the level sets $Y_1=\{|\xi|\ge a\}$ 
and $Y_2=\{|\eta|\ge a\}$ satisfy 
$\int_{X\setminus Y_1}|\xi|=\int_{X\setminus Y_2}|\eta|<\var$. Then $\mu(Y_1)=\mu(Y_2)<\infty$. The functions
$$
\xi'=\begin{cases} \xi & \text{ on } Y_1\\ 0 &\text{ on } X\setminus Y_1\end{cases}\qquad\text{and}\qquad
 \eta'=\begin{cases}\eta &\text{ on } Y_2\\0 & \text{ on } X\setminus Y_2\end{cases}
$$
are also equidistributed. Construct a connected open $X'\subset X$ of finite measure containing $Y_1\cup Y_2$. 
By what we have  proved so far, there is a
$g\in\Ham_0(\omega|X')$ such that $\int_{X'}|\xi'-\eta'\circ g|<3\var$. Extend $g$ to all of $X$ by identity on 
$X\setminus X'$. 
Denoting this extension also by $g$, we have
$$
\int|\xi-\eta\circ g|\le\int |\xi'-\eta'\circ g|+\int|\xi-\xi'|+\int|\eta-\eta'|< 3\var+\var+\var=5\var.
$$

To finish the proof, we let $\var=1/k$ and $g=g_k$, $k\in\bN$, and obtain the sequence sought.
\end{proof}

\begin{proof}[Proof of Theorem 3.1] 
Consider $\cal{A}\subset\bR\times C^\infty(X)$ of Lemma 2.1:
\[
p(\xi)=\sup \Big\{a+\int f\xi : (a,f)\in \cal{A}\Big\}.
\]
Suppose $\xi, \eta\in C^\infty(X)$ are equidistributed, and let $g_k$ be as in Lemma 3.2. With any $(a,f)\in\cal{A}$
\[ 
p(\eta)=p(\eta\circ g_k)\ge a+\int(\eta\circ g_k)f\to a+\int f\xi\quad \text{ as } k\to \infty.
\]
Taking sup over $(a,f)\in \cal{A}$,  $p(\eta)\ge p(\xi)$ follows, and in fact $p(\xi)=p(\eta)$ by symmetry.
\end{proof}

\section{Proof of the first part of Theorem 1.2}  

This is what the first part says:

\begin{thm}  
If $(X,\omega)$ is a connected compact symplectic manifold, any continuous, convex, $\Ham(\omega)$ invariant function $p:C^\infty(X)\to\bR$ is continuous in the sup norm topology on $C^\infty(X)$. 
\end{thm}

We will use the following standard fact:

\begin{lem}  
Let $V$ be a locally convex topological vector space over $\bR$. If $p:V\to\bR$ is convex and bounded above 
on some open $U\subset V$, then it is continuous on $U$.
\end{lem}
\begin{proof}
We can assume $U$ is convex. Say, we want to prove continuity at $0\in U$. Let $s=\sup_U p <\infty$. With 
$0<\lambda<1$ and $v\in (\lambda U)\cap (-\lambda U)$ convexity implies
\begin{equation*}
\begin{rcases}
&p(v)-p(0)\le\lambda(p(v/\lambda)-p(0))\le \lambda(s-p(0))\\
&p(0)-p(v)\le\lambda(p(-v/\lambda)-p(0))\le \lambda(s-p(0))
\end{rcases}
\to 0
\end{equation*}
when $\lambda\to 0$, as needed.
\end{proof}

The key to the proof of Theorem 4.1 is the following.

\begin{lem}  
Let $\cal{F}\subset L^1(X)$ be a $\Ham(\omega)$ invariant family of functions. If for every $\xi\in C^\infty(X)$
\begin{equation}   
\sup_{f\in\cF} \int_X f\xi\,d\mu < \infty,
\end{equation}
then $\sup_{f\in\cF} \int_X |f|\,d\mu < \infty$.
\end{lem}

This is not hard to show and will suffice to prove Theorem 4.1; but later we will need a more precise 
statement, whose proof is just a little more involved. Let $\xi^+=\max(\xi,0)$ and $\xi^-=\max(-\xi,0)$ denote the positive 
and negative parts of functions $\xi\colon X\to\bR$. If $E\subset X$ is measurable, write $\fint_E\xi$ for the
 average $\int_E\xi/\mu(E)$ of an integrable function. If $\mu(E)=0$, we let $\fint_E\xi=0$.

\begin{lem} 
Let $f\in L^1(X)$, $\xi\in L^\infty(X)$, and $S,T\subset X$ be of equal measure. If 
$\xi\ge 0$ on $T$  and $\xi\le 0$ on $X\setminus T$, then
\begin{equation}
\sup\Big\{\int_X (f\circ g)\xi\colon g\in\Ham(\omega)\Big\} 
\ge\fint_S f\int\xi^+-\fint_{X\setminus S} f\int\xi^-.
\end{equation}
\end{lem}

First we show how this implies Lemma 4.3.

\begin{proof}[Proof of Lemma 4.3]
We can assume $\mu(X)=1$. Let $M(\xi)$  denote the  left hand side of (4.1).
Fix a nonnegative $\xi\in C^\infty(X)$ that is not identically $0$, but $T'=\{\xi>0\}$ has measure $\le1/2$. 
Let $f\in\cF$. Suppose first that $S=\{f\ge 0\}$ has measure $\ge1/2$, and choose $T\supset T'$ so
that $\mu(S)=\mu(T)$. By Lemma 4.4
$M(\xi)\ge \fint_S f\int\xi$, hence
\[
\int f^+\le 2M(\xi)\Big/\int\xi.
\]
If, instead of $S$, $\{f\le 0\}$ has measure $\ge1/2$, Lemma 4.4 implies in the same way that 
$\int f^-\le 2M(-\xi)/\int\xi$. Since $|\int f^+-\int f^-|=|\int f|\le M(1)+M(-1)$, in both cases we obtain a bound for 
$\int|f|=\int f^++\int f^-$, as
claimed.

\end{proof}

Given $f\in L^1(X)$, we will write conv$_1(f)$ for the closure, in the $L^1(X)$ topology, of the convex hull of the
orbit of $f$ under  $\Ham(\omega)$. In light of Lemma 3.2 this is the same as the closed convex hull of all strict
rearrangements of $f$.
To prove Lemma 4.4 we need the following.
\begin{lem}
If $f\in L^1(X)$ and $E\subset X$ has positive measure, then the function
\[
f'=\begin{cases}\fint_E f\quad&\text{on } E\\f&\text{on } X\setminus E
\end{cases}
\]
is in conv$_1(f)$.
\end{lem}

\begin{proof}
If two functions $f,h\in L^1(X)$ are at $L^1$ distance $\le\var$, then their $\Ham(\omega)$ orbits are at Hausdorff distance
$\le\var$, and so are therefore conv$_1(f)$ and conv$_1(h)$. Hence, given $E$, if the lemma holds for a sequence $f=f_k$, 
$k=1,2,\ldots$, and $f_k\to f_0$ in $L^1$, then the lemma will hold for $f_0$ as well.

Now suppose that $E$ is the disjoint union of $E_j$, $j=1,\ldots, m$, of equal measure, and $f=c_j$ is constant on 
each $E_j$.
If $\sigma$ is a permutation of $1,\dots,m$, define $f_\sigma\in L^1(X)$ by
\[
f_\sigma=c_{\sigma(j)} \quad\text{on } E_j,\qquad f_\sigma= f\quad\text{on } X\setminus E.
\]
As a strict rearrangement of $f$, by Lemma 3.2 $f_\sigma$ is in the closure of the $\Ham(\omega)$ orbit of $f$. Therefore
\[
f'=\sum_\sigma  f_\sigma/m!
\]
is indeed in conv$_1(f)$. Since any $f\in L^1(X)$ is the limit of functions of the above type, the claim follows. 

\end{proof}
\begin{proof}[Proof of Lemma 4.4]
Write $\chi_A$ for the characteristic function of a set $A$. By Lemma 3.2 there is a sequence $g_k\in\Ham(\omega)$ 
such that $\chi_S\circ g_k\to \chi_T$ in $L^1$. Two applications of Lemma 4.5 give that
\[
f'=\begin{cases}\fint_S f&\text{on } S\\ \fint_{X\setminus S} f&\text{on } X\setminus S\end{cases}\quad\text{and so}\quad
f''=\lim_k f'\circ g_k=\begin{cases}\fint_S f&\text{on }T\\\fint_{X\setminus S}f&\text{on } X\setminus T\end{cases}
\]
are in conv$_1(f)$.  Lemma 4.4. follows, since the left hand side in (4.2) is 
\[
\ge \int f''\xi=
 \fint_S f\int_T\xi+\fint_{X\setminus S}f\int_{X\setminus T}\xi=
\fint_S f\int\xi^+-\fint_{X\setminus S}f\int\xi^-.
\]
\end{proof}
\begin{proof}[Proof of Theorem 4.1]
If a function is continuous in the $\sup$ norm topology, we will say it is $\|\ \ \|_\infty$--continuous, and use similar terminology for
other topological notions. First assume that $p$ of the theorem is positively homogeneous as well. By Lemma 2.1 there is 
a family $\cF\subset L^1(X)$ such that
\begin{equation}  
p(\xi)=\sup\Big\{\int f\xi : f\in\cF\Big\}.
\end{equation}
If we replace $\cF$ by its $\Ham(\omega)$ orbit, the supremum in (4.3) will not change, for 
$$
\int (f\circ g)\xi=\int (\xi\circ g^{-1})f\le p(\xi\circ g^{-1})=p(\xi)\qquad\text{if } f\in\cF, \   g\in\Ham(\omega).
$$  
Therefore we may assume that the family $\cF$ in (4.3) is already invariant under $\Ham(\omega)$.
 Hence Lemma 4.3 gives 
$\sup_\cF \int |f| <\infty$. This implies $p$ is bounded on $\| \ \ \|_\infty$--bounded subsets of $C^\infty(X)$, and by Lemma 4.2 it is $\| \ \ \|_\infty$--continuous.

For general $p$, pick a number $c>p(0)$ and consider the Minkowski functional $q$ of the convex set $\{p<c\}$ (see e.g. \cite[pp. 315-317]{Sc}), 
\[q(\xi)=\inf\{\lambda\in(0,\infty) : p(\xi/\lambda)< c\} \in [0,\infty).\]
This is a convex, positively homogeneous, strict rearrangement invariant function, that is continuous---because locally bounded---in the topology of $C^\infty(X)$. By what we have already proved, it is $\| \ \|_\infty$--continuous, in particular, the set 
$U_c=\{q<1\}\supset\{p<c\}$ is $\| \ \|_\infty$--open. If $\xi\in U_c$ then $p(\xi/\lambda)< c$ with some $\lambda < 1$. Also 
$p(0)<c$. As $\xi$ is a point on the segment connecting 0, $\xi/\lambda$, convexity implies $p(\xi)<c$. Thus $p$ is 
bounded above on the $\| \ \|_\infty$--open set $U_c$, and by Lemma 4.2 it is continuous there. The theorem follows since $\bigcup_c U_c=C^\infty(X)$.
\end{proof}

\section{Extending convex functions}

The above ideas can be developed to prove that $p$ can be extended to $C(X)$ and, under an additional assumption, to
the Banach space $B(X)$ of bounded Borel functions, with the supremum norm. 
(Thus $L^\infty(X)$ is a quotient of $B(X)$, but
$B(X)$  is more natural to use in our setting.) 
\begin{defn}
If $V\subset B(X)$ is a vector subspace, we say that a function $p\colon V\to\bR$ is strongly continuous if $p(\xi_k)$ is 
convergent whenever
$\xi_k\in V$ is an almost everywhere convergent sequence of uniformly bounded functions. 
\end{defn}

The limit $\lim p(\xi_k)$ depends only on 
$\lim\xi_k=\xi$, since two such 
sequences can be combined into one sequence, converging to $\xi$.
\begin{thm}
Any continuous, convex, $\Ham(\omega)$ invariant $p\colon C^\infty(X)\to\bR$
has a unique continuous extension to $C(X)$; this extension is convex and 
$\Ham(\omega)$ (hence strict rearrangement)
invariant. If $p$ is  strongly continuous, then it has a unique strongly continuous  extension 
$q:B(X)\to\bR$. This extension is convex, and invariant under strict rearrangements.
\end{thm}

Since 
$C^\infty(X)$ is dense in $C(X)$, and $p$ is known to be continuous in supremum norm, for the first part of Theorem 5.2 
one only needs to prove that a continuous extension exists. This is a special case of the following: 

\begin{lem}
Let $W$ be a locally convex topological vector space over $\bR$, $V\subset W$ a dense subspace. Any continuous, 
convex  $p:V\to\bR$ can be  extended to a continuous $q:W\to\bR$.
\end{lem}
\begin{proof}
First we show that any $w\in W$ has a convex neighborhood $U$ such that $p$ is bounded on $V\cap U$. By continuity, 
there certainly is a symmetric, convex 
neighborhood $U_0\subset W$ of $0$ such that $p$ is bounded on $V\cap 4U_0$. Now $w+2U_0$ is a neighborhood of
$w$, and if $v_1\in V$ is sufficiently close to $w$, then $U=v_1+2U_0$ is also. For any $v\in V\cap U$  convexity implies
\[
2p(v)\le p(2v_1)+p\big(2(v-v_1)\big).
\]
Since $v-v_1\in 2U_0$, the right hand side is bounded as $v$ varies in $V\cap U$. Thus $p$ is bounded above on 
$V\cap U$. 
But then  $p(v)+p(2v_1-v)\ge 2p(v_1)$ gives that $p$ is also bounded below. Set $s=\sup_U |p|$.

We let $U'=v_1+U_0$ and show that $p$ is uniformly continuous on $V\cap U'$.  For suppose $\lambda\in(0,\infty)$.
If $u,v\in V\cap U'$ and $v-u\in U_0/\lambda$, then $v+\lambda(v-u)\in v_1+U_0+U_0=U$, hence by convexity
\[ 
p(v)-p(u)\le\frac{p\big(v+\lambda(v-u)\big)-p(u)}{1+\lambda}\le\frac {2s}{1+\lambda}.
\]
Since the roles of $u,v$ are symmetric, this indeed proves locally uniform continuity; which in turn implies continuous extension.
 \end{proof}
 \begin{proof}[Proof of Theorem 5.2]
We have already seen that the first half of the theorem follows from Lemma 5.3. As to the uniqueness of 
extension to $B(X)$, we 
note that Lusin's theorem implies that any $\xi\in B(X)$ is the a.e. limit of a uniformly bounded sequence
of continuous, hence also of
smooth functions $\xi_k$. Therefore at $\xi$ the extension of 
$p$ must take the  value $\lim_k p(\xi_k)$, so it is unique. What remains is to construct the required extension $q$. 

If $\xi\in B(X)$, predictably 
we let $q(x)=\lim_kp(\xi_k)$, where the uniformly bounded sequence $\xi_k\in C^\infty(X)$ converges
 to $\xi$ a.e. As we saw, this is 
independent of the choice of the sequence $\xi_k$. Clearly $p=q$ on $C^\infty(X)$. If uniformly bounded 
$\eta_k\in C^\infty(X)$ converge to $\eta\in B(X)$ a.e., and 
$\lambda\in [0,1]$, then 
\begin{multline*}
q\big(\lambda \xi+(1-\lambda)\eta\big)=\lim_k p\big(\lambda \xi_k+(1-\lambda)\eta_k\big)\\ 
\le\lim_k \lambda p(\xi_k)+(1-\lambda)p(\eta_k)=\lambda q(\xi)+(1-\lambda)q(\eta),
\end{multline*}
i.e., $q$ is convex. It is also strongly continuous. For this it suffices to show that if uniformly bounded $\xi_k\in B(X)$ 
converge to $\xi$ a.e., then a subsequence of $q(\xi_k)$ tends to $q(\xi)$. By dominated convergence,
\begin{equation}
\lim_k\int|\xi_k-\xi|=0.
\end{equation}
Let each $\xi_k$ be the a.e. limit of a uniformly bounded sequence $\xi_k^i\in C^\infty(X)$, as $i\to\infty$. We can arrange that 
the double sequence $\xi_k^i$ is also uniformly bounded. Thus 
$\lim_{i\to\infty} p(\xi_k^i)=q(\xi_k)$. For each $k$ choose $i=i_k$ so that $\eta_k=\xi^i_k$ satisfies
\begin{equation}
|p(\eta_k)-q(\xi_k)|<1/k,\qquad \int|\eta_k-\xi_k|<1/k.
\end{equation}
In view of (5.1) $\lim_k\int|\eta_k-\xi|=0$, so a subsequence  $\eta_{k(j)}$ converges to $\xi$ a.e. Hence, by (5.2)
\[
q(\xi)=\lim_j p(\eta_{k(j)})=\lim_j q(\xi_{k(j)}),
\]
as needed.

Finally, to show that $q$ is invariant under strict rearrangements, consider equidistributed $\xi,\eta\in B(X)$. By
Lemma 3.2 there are $g_k\in\Ham(\omega)$ such that $\int|\eta-\xi\circ g_k|\to 0$ as $k\to\infty$. Choose uniformly
bounded $\xi_k\in C^\infty(X)$ converging to $\xi$ a.e. In particular, $\lim_k\int |\xi_k-\xi|= 0$. Then
\[
\lim_k\int|\xi_k\circ g_k-\eta|\le\limsup_k\int|(\xi_k-\xi)\circ g_k|+\limsup_k\int|\xi\circ g_k-\eta|=0.
\]
Again, this means that a subsequence of $\xi_k\circ g_k$ converges a.e. to $\eta$, whence
\[
q(\xi)=\lim_k p(\xi_k)=\lim_k p(\xi_k\circ g_k)=q(\eta),
\]
which proves that $q$ is indeed invariant under strict rerrangements.
\end{proof}

Here is the last theorem in this section.

\begin{thm}
If a strict rearrangement invariant convex $p:B(X)\to\bR$ is  strongly continuous, then it is Lipschitz continuous
on bounded sets.
\end{thm}
\begin{lem}
There is a continuous $\theta:X\to[0,\mu(X)]$ that is smooth away from the preimage of finitely many $t\in[0,\mu(X)]$, 
and that preserves measure (the target is endowed with Lebesgue measure).

\end{lem}
\begin{proof}
If $\zeta\in C^\infty(X)$ is a Morse function, its reverse distribution function
\[
\lambda(t)=\mu(\zeta<t),\qquad t\in[\min\zeta,\max\zeta],
\]
is continuous, strictly increasing, and smooth away from the set $C$ of critical values of $\zeta$. It is a homeomorphism 
$[\min\zeta,\max\zeta]\to[0,\mu(X)]$, and a diffeomorphism away from $C$. The function 
$\theta=\lambda\circ\zeta$ will therefore do, as
\[
\mu(\theta<s)=\mu(\zeta<\lambda^{-1}(s))=\lambda(\lambda^{-1}(s))=s,\qquad s\in[0,\mu(X)].
\]
\end{proof}
We will need the notion of decreasing rearrangement of a measurable $\xi:X\to\bR$. It is the decreasing, say, upper 
semicontinuous function
$\xi^\star:[0,\mu(X)]\to\bR$ that is equidistributed with $\xi$. Thus $\mu(s\le \xi\le t)$ is equal to the length of the maximal
interval on which $s\le\xi^\star\le t$. In particular, 
\begin{equation}
\mu(\xi\ge\xi^\star(s))=s.
\end{equation}
The upper semicontinuity requirement translates to
left continuity of the decreasing function $\xi^\star$, which differs from the more usual  convention of right continuity, but
the difference is inconsequential. Obviously, with $\theta$ of Lemma 5.5 $\xi$ and $\xi^\star\circ\theta$ are
equidistributed. 
\begin{lem}   
If $\xi\in C(X)$, then $\xi^\star$ is continuous.
\end{lem}

\begin{proof}
Since $\xi^\star$ is always u.s.c., i.e., left continuous, all we need to show is that if $s_j\in[0,\mu(X)]$ decreases to $s$, then $\lim_j\xi^\star(s_j)$ cannot be $>\xi^\star(s)$. Suppose it were, and let $\xi^\star(s)<\alpha<\beta<\lim_j \xi^\star(s_j)$. 
Then $\xi^{-1}(\alpha,\beta)\subset X$ would be a nonempty open subset, of positive measure, contradicting (cf.(5.1))
\[\mu(\xi\ge\xi^\star(s_j))=s_j\to s=\mu(\xi\ge\xi^\star(s)).\]
\end{proof}

\begin{proof}[Proof of Theorem 5.4]
Let $\theta$ be as in Lemma 5.5. We start by showing that $p$ is bounded on bounded sets. Otherwise 
there would be a bounded sequence $\xi_k\in B(X)$ such that  $|p(\xi_k)|\to \infty$. The 
decreasing rearrangements $\xi_k^\star$ are uniformly bounded,
hence by Helly's theorem contain a pointwise convergent subsequence. But along that subsequence $\xi_k^\star\circ\theta$
converges pointwise and therefore by strong continuity
\[
p(\xi_k)=p(\xi_k^\star\circ\theta)
\]
also converges, a contradiction. 

Now boundedness on bounded sets implies Lipschitz continuity on bounded sets. 
For suppose $\xi\neq\eta$ have norm $\le R$, 
and let $\rho$ be the unit vector in the direction of $\xi-\eta$. With $M=\sup_{||\zeta||_\infty\le R+1}|p(\zeta)|$, by convexity
\[
\frac{p(\xi)-p(\eta)}{||\xi-\eta||_\infty}\le\frac{p(\xi+\rho)-p(\eta)}{||\xi+\rho-\eta||_\infty}\le 2M.
\]
The roles of $\xi,\eta$ being symmetric, we obtain Lipschitz continuity.
\end{proof}

\section{Proof of Theorem 1.3}

To simplify notation, we will assume $\mu(X)=1$. 
By Lemma 2.1 a $\Ham(\omega)$ invariant convex, continuous, $p\colon C^\infty(X)\to\bR$ can be written
\begin{equation}
p(\xi)=\sup\Big\{a+\int f\xi : (a,f)\in\cA\Big\}
\end{equation}
with a family $\cA\subset \bR\times C^\infty(X)$, that can be chosen convex and invariant under $\Ham(\omega)$. 
The possible behaviors of $p$ described in Theorem 1.3 are determined by whether all functions $f$ that occur in $\cA$ 
are
constant or not. 

If in $\cA$ only constant functions occur, then (6.1) gives $p(\xi)=p(\int\xi)$. 
Henceforward we will assume $\cA$ contains a pair $(a,f)$ with a nonconstant function $f$. According to (ii) of 
Theorem 1.3, we must 
estimate $p(\xi)$ from below with the $L^1$ norm of $\xi\in C^\infty(X)$. We do this do in a somewhat greater generality, that we will need in
the next section. 

\begin{lem}
Suppose $\cA\subset\bR\times L^1(X)$ is convex and invariant under $\Ham(\omega)$. For
$\xi\in L^\infty(X)$ let $q(\xi)=\sup_{(a,f)\in\cA}a+\int f\xi$. If $\cA$ contains a pair $(a,f)$ with $f$ 
nonconstant, then there are
$ a_0\in \bR$ and  $b\in(0,\infty)$
such that
\[
q(\xi)\ge a_0+b\int |\xi|\quad\text{ if }\quad\begin{cases} \int \xi=0,\quad{\text{or}}\\ 
\int \xi\ge0{\text\quad\text{and }\ \lim_{\bR\ni\lambda\to\infty}q(\lambda)>q(0), \quad\text{or}} \\
 \int\xi\le 0\quad\text{and }\ \lim_{\bR\ni\lambda\to-\infty}q(\lambda)>q(0).\end{cases}
\]
If $\cA\subset \{0\}\times L^1(X)$, then $a_0$ can be chosen $0$.
\end{lem}
\begin{proof}
Fix $(a,f)\in\cA$ with $f$ nonconstant. If $\alpha\in(0,1]$ let
\begin{equation}
s_\alpha=s_\alpha(f)=\sup_{\mu(E)=\alpha}\fint_Ef,\qquad 
i_\alpha=i_\alpha(f)=\inf_{\mu(E)=\alpha}\fint_E f,
\end{equation}
and let $s_0=\text{ess\,sup\,} f$, $i_0=\text{ess\,inf\,} f$. For every $\alpha>0$ there is an $S=S_\alpha\subset X$ 
of measure  $\alpha$ for which  $\fint_Sf=s_\alpha$. Indeed,  consider 
$$
u=\inf\big\{t\in\bR\colon \mu\{f>t\}\le \alpha\big\}.
$$
Since $\mu\{f>u\}\le\alpha\le\mu\{f\ge u\}$, any set $S$ of measure $\alpha$ sandwiched between $\{f>u\}$ and 
$\{f\ge u\}$ will provide the sup in (6.2). Similarly, $S'=X\setminus S$, of measure $1-\alpha$, satisfies
$i_{1-\alpha}=\fint_{S'}f$. This implies that $s_\alpha>i_{1-\alpha}$.
From the absolute continuity of $fd\mu$ with respect to $d\mu$ we deduce
that $s_\alpha,i_\alpha$ are continuous functions of $\alpha>0$; continuity trivially holds at $\alpha=0$ as well. Hence
\begin{equation}
2c=2c(f)=\min_{0\le\alpha\le1}(s_\alpha-i_{1-\alpha})>0, \quad 
2m=2m(f)=\max_{0\le\alpha\le 1}|s_\alpha|+|i_{1-\alpha}|<\infty.
\end{equation} 

Consider a $\xi\in L^\infty(X)$ and let $T=\{\xi\ge 0\}$. With $\alpha=\mu(T)$ and $S=S_\alpha$ as above,  Lemma 4.4 implies
\begin{equation}
q(\xi)\ge a+s_\alpha \int\xi^+-i_{1-\alpha}\int \xi^-=a+
\frac{s_\alpha-i_{1-\alpha}}2\int|\xi|+\frac{s_\alpha+i_{1-\alpha}}2\int\xi
\end{equation}
(even if $\alpha=0$). When $\int\xi=0$, by (6.3) we obtain $q(\xi)\ge a+c\int|\xi|$. 

Next suppose that $\lim_{\lambda\to\infty}q(\lambda)>q(0)$. There are $\lambda>0$ and $(a_1,f_1)\in\cA$ with
 $a_1+\int f_1\lambda>q(0)\ge a_1$; hence
$\int f_1>0$. Because $\cA$ is convex, we can arrange that
our fixed $(a,f)\in\cA$ already satisfies $\int f>0$. Let
$b=s_1c/(s_1+m)$. We will show that if $\int\xi\ge 0$, then $q(\xi)\ge a+b\int|\xi|$. Note that the constant 
function $f'=\int f$ is in conv$_1(f)$ according to 
Lemma 4.5, and $(a,f')$ is in $\cA$. Hence
$q(\xi)\ge a+\int f'\xi=a+s_1\int\xi$. By (6.4) $q(\xi)\ge a+c\int|\xi|-m\int\xi$. Combining these two we can eliminate 
$\int\xi$ and obtain
\[
mq(\xi)+s_1q(\xi)\ge (m+s_1)a+s_1c\int|\xi|,
\]
as needed. Finally, if $\lim_{\lambda\to-\infty}q(\lambda)>q(0)$, we choose $(a,f)\in\cA$ such that $f$ is nonconstant and
$\int f<0$. Letting  
$b=c(f)|s_1(f)|/(|s_1(f)|+m(f))$ we can similarly prove $q(\xi)\ge a+b\int|\xi|$ whenever $\int\xi\le 0$. This completes the 
proof of the lemma, and also of the theorem.
\end{proof}
\section{Proof of Theorem 1.4}

This was the theorem:
\begin{thm}
Given a $\Ham(\omega)$ invariant continuous norm $p$ on $C^\infty(X)$, there is a rearrangement invariant Banach 
function space on $X$ whose norm, restricted to $C^\infty(X)$, is equivalent to $p$.
\end{thm}

We will get to the notion of rearrangement invariant Banach spaces shortly, but first we 
formulate a few auxiliary results that we will  need. Let us say that two functions $\phi,\psi:X\to\bR$ 
are similarly ordered if $\big(\phi(x)-\phi(y)\big)\big(\psi(x)-\psi(y)\big)\ge 0$ for all $x,y\in X$. Put it differently, 
$\phi(x)>\phi(y)$ should imply $\psi(x)\ge\psi(y)$. In spite of what the language may suggest, this is not an equivalence
relation (all functions are similarly ordered as a constant). However, it is true that if $\phi$ and $\psi$ are similarly ordered,
and $U:\bR\to\bR$ is increasing, then $\phi$ and $U\circ\psi$ are also similarly ordered.

We will write $\phi\sim\psi$ for measurable functions $X\to\bR$ if they are equidistributed. The following lemma in one 
form or another is known and, like Lemmas 7.3, 7.4,  7.5, holds in any finite measure space $(X,\mu)$ without atoms.

\begin{lem}
Let $\phi_0\in L^1(X)$ be bounded below and $\psi_0\in L^\infty(X)$. 

(a) $\sup_{\phi\sim\phi_0}\int\phi\psi_0=\sup_{\psi\sim\psi_0}\int\phi_0\psi$.

(b) The suprema in (a) are attained, by $\phi$ and $\psi$ that are similarly ordered as $\psi_0$ and $\phi_0$.

(c) $\int\phi\psi$ is independent of the choice of $\phi\sim\phi_0$, $\psi\sim\psi_0$, as long as $\phi,\psi$ are similarly
ordered.
\end{lem}
\begin{proof}(b)  That the suprema are attained, at least when $\phi_0,\psi_0\ge 0$, is proved in [BS, Chapter 2, Theorems
2.2 and 2.6]. 
The general result follows upon adding a constant to the functions. The proof in [BS, pp. 49-50], say, for the first 
supremum in (a),  
proceeds by first considering simple
$\phi_0$ and representing the maximizing $\phi$ by an explicit formula, then passing to a limit. 
The formula shows that $\phi$ and $\psi_0$ are similarly ordered when $\phi_0$ is simple; but similar ordering is 
preserved under pointwise limits, and must hold in general.

(c) Again, first assume that $\psi_0$ is simple, and takes values $a_1<a_2<\dots<a_k$. Let $A_j=\{x:\psi(x)=a_j\}$. If
necessary, we can change the values of $\phi,\psi$ on a set of zero measure to arrange that each $\mu(A_j)>0$. Let
\[
m_j=\inf_{A_j}\phi, \qquad M_j=\sup_{A_j}\phi.
\]
If $x\in A_j$ and $y\in A_{j+1}$, then $\psi(x)<\psi(y)$ and $\phi(x)\le\phi(y)$, hence
\begin{equation}
\ldots\le m_j\le M_j\le m_{j+1}\le\dots
\end{equation}
It follows that the set $B_j=\{x:m_j<\phi(x)<M_j\}$ is included in $A_j$. With 
$C_j=\{x\in A_j:\phi(x)=m_j\}$ and $D_j=\{x\in A_j: \phi(x)=M_j\}$ therefore
\[
\int\phi\psi=\sum_j a_j\int_{A_j}\phi=\sum_ja_j
\begin{cases}\int_{B_j}\phi+m_j\mu(C_j)+M_j\mu(D_j)&\text{ if } m_j<M_j\\m_j\mu(A_j)&\text{ if } m_j=M_j.\end{cases}
\]
We will show that each term on the right is determined by $\phi_0,\psi_0$.

To start,
\begin{equation}
m_j=\sup\Big\{m:\mu(\phi\ge m)\ge\mu\Big(\bigcup_{i=j}^k A_i\Big)\Big\},
\end{equation}
because by (7.1)
\[
\phi\le m_j\quad\text{on } \bigcup_{i=1}^{j-1}A_i,\qquad\phi \ge m_j\quad\text{on } \bigcup_{i=j}^k A_i.
\]
Since $\mu(\phi\ge m)=\mu(\phi_0\ge m)$ and $\mu\big(\bigcup_j^k A_i\big)=\mu(\psi_0\ge a_j)$, (7.2) shows that the 
$m_j$ are 
determined by $\phi_0,\psi_0$; and so are the $M_j$. It follows that
\[
\mu(B_j)=\mu(m_j<\phi_0<M_j)\quad\text{ and } \quad \int_{B_j}\phi=\int_{\{m_j<\phi_0<M_j\}}\phi_0
\]
are also determined by $\phi_0,\psi_0$. Next, $\mu(A_j)=\mu(\psi_0=a_j)$. Finally, if $j$ is such that $m_j< M_j$, then in
light of (7.1) $C_j=(\phi\le m_j)\setminus \bigcup_1^{j-1} A_i$,
\begin{align*}
\mu(C_j)&=\mu(\phi \le m_j)-\mu\big(\bigcup_1^{j-1}A_i\big)=\mu(\phi_0\le m_j)-\mu(\psi_0<a_j) \quad\text{ and}\\
\mu(D_j)&=\mu(A_j)-\mu(B_j)-\mu(C_j).
\end{align*}

This proves (c) for a simple $\psi_0$. To finish the proof, consider a general $\psi_0$. Let $\lfloor\ \ \rfloor$ denote integer part 
and for $k\in\bN, t\in\bR$ let $U_k(t)=\lfloor kt\rfloor/k$, an increasing function of $t$. By what we have proved
$\int (U_k\circ\psi)\phi$ is determined by $\phi_0,\psi_0$, hence so is (by the dominated convergence 
theorem)
\[
\int\phi\psi=\lim_{k\to \infty}\int (U_k\circ\psi)\phi.
\]

(a) now follows from (b) and (c).
\end{proof}
\begin{lem}
If $\phi\in L^1(X)$ and $\psi\in L^\infty(X)$ are similarly ordered, then $\int\phi\psi\ge\fint\phi\int\psi$.
\end{lem}
\begin
{proof} This is Chebishev's  integral inequality.  See for the discrete
version of the inequality---from which the lemma follows---p. 43 in \cite{HLP}, and also p. 168.
\end{proof}

\begin{lem}
If $\phi_0,\psi\in L^\infty(X)$, then
\begin{equation}
\sup_{\phi\sim\phi_0}\int|\phi|\psi\le\sup_{\phi\sim\phi_0}\int\phi\psi+
\sup_{\phi\sim\phi_0}\int(-\phi)\psi+
\fint|\phi_0|\int\psi.
\end{equation}
\end{lem}
\begin{proof}
First we estimate $\int\phi^+\psi$. By Lemma 7.2 we can choose $\phi_1\sim\phi_0$, similarly ordered as $\psi$, that
realizes $\sup_{\phi\sim\phi_0}\int\phi\psi$. It follows that $\phi_1^+$, a composition of $\phi_1$ with an increasing 
function, is also similarly ordered as $\psi$. Using Lemma 7.2 once more we obtain
\[
\sup_{\phi\sim\phi_0}\int\phi^+\psi=\int\phi_1^+\psi=\int\phi_1\psi+\int\phi_1^-\psi.
\]
As $-\phi_1^-$ and $\psi$ are similarly ordered, Lemma 7.3 gives $-\int\phi_1^-\psi\ge-\fint\phi_1^-\int\psi$, and so
\begin{equation}
\sup_{\phi\sim\phi_0}\int\phi^+\psi\le\int\phi_1\psi+\fint\phi_1^-\int\psi=\sup_{\phi\sim\phi_0}\int\phi\psi+\fint\phi_0^-\int\psi.
\end{equation}
Replacing $\phi_0$ with $-\phi_0$,
\begin{equation}
\sup_{\phi\sim\phi_0}\int\phi^-\psi\le\sup_{\phi\sim\phi_0}\int(-\phi)\psi+\fint\phi_0^+\int\psi,
\end{equation}
and (7.3) follows by adding (7.4) and (7.5).
\end{proof}

\begin{lem}
If $f_0,\xi\in L^\infty(X)$ then $\sup_{f\sim f_0}\int|f\xi|\le 4\sup_{f\sim f_0}|\int f\xi|+3\fint|f_0|\int|\xi|$.
\end{lem}
\begin{proof} 
Let us start with a simple $\xi$.  Lemma 7.4, with $\phi_0=f_0$, $\psi=|\xi|$  gives 
\begin{equation}
\sup_{f\sim f_0}\int |f\xi|\le 2\sup_{f\sim f_0}\Big|\int f|\xi|\Big|+\fint|f_0|\int|\xi|.
\end{equation}

By Lemma 7.2 
\begin{equation}
\sup_{f\sim f_0}\int f|\xi|=\sup_{\zeta\sim|\xi|}\int f_0\zeta.
\end{equation}
Any $\zeta\sim |\xi|$ can be written as $\zeta=|\eta|$ with $\eta\sim \xi$. Indeed, suppose $\xi$ takes distinct values
$a_1,\dots,a_k$. If for some $i$ there is no $j$ with $a_i=-a_j$, we let $\eta\equiv a_i$ on the set $(\zeta=|a_i|)$. If
for some $i$ there is a (necessarily unique) $j$ with $a_i=-a_j$, for each such pair we divide the set $(\zeta=|a_i|=|a_j|)$ in
 two parts, of measures $\mu(\xi=a_i)$, $\mu(\xi=a_j)$, and define $\eta\equiv a_i$ on the former, $\eta\equiv a_j$ on the latter.
 
 Hence, applying Lemma 7.4 again, this time with 
 $\phi_0=\xi$, $\psi=f_0$, we obtain
 \[
 \sup_{\zeta\sim|\xi|}\int f_0\zeta=\sup_{\eta\sim \xi}\int f_0|\eta|\le 2\sup_{\eta\sim \xi}\Big|\int f_0\eta\Big|+\fint f_0\int|\xi|.
 \]
 In light of (7.7) and Lemma 7.2 therefore
 \[
 \sup_{f\sim f_0}\int f|\xi|\le 2\sup_{f\sim f_0}\Big|\int f\xi\Big|+\fint |f_0|\int|\xi|.
 \]
Substituting this, and its counterpart with $f_0$ replaced by $-f_0$, into (7.6) gives  the lemma, when $\xi$ is simple. A general
$\xi$ can be uniformly approximated by simple functions $\xi_m$, and knowing the estimate for each $\xi_m$ gives the 
estimate for $\xi$ in the limit.

\end{proof}

\begin{proof}[Proof of Theorem 7.1] By Lemma 2.1
$
p(\xi)=\sup\{\int f\xi\colon f\in\cF\}
$
with a family $\cF\subset L^\infty(X)$, that we can choose to be invariant under $\Ham(\omega)$. Because of Lemma 3.2
we can even choose it to be invariant under strict rearrangements. For any measurable 
$\zeta:X\to[-\infty,\infty]$ define
\[
q(\zeta)=\sup\Big\{\int|f \zeta|\colon f\in\cF\Big\}\in[0,\infty],
\]
and let $B=\{\zeta\colon q(\zeta)<\infty\}$, $\|\ \ \|=q|B$. Some obvious properties of $q$ are: it is positively homogeneous, 
$q(\eta+\zeta)\le q(\eta)+q(\zeta)$, and $|\eta|\le|\zeta|$ a.e. implies $q(\eta)\le q(\zeta)$. If $q(\zeta)=0$ then $\zeta=0$
a.e. on any set where some $f\in\cF$ is nonzero; since $\cF$ is invariant under strict rearrangements, this simply means 
$\zeta=0$ a.e. By Lemma 4.3  $\sup_{f\in\cF}\int|f|<\infty$, hence $L^\infty(X)\subset B$. 
Furthermore, $q$ is invariant under all rearrangements, strict or not; this also implies by Lemma 6.1, with a suitable $b>0$, 
\begin{equation}
q(\zeta)\ge b\int|\zeta|
\end{equation}
if $\zeta\in L^\infty(X)$.

Following Bennett--Sharpley's definition [BS, pp. 2, 59], $(B,\|\ \ \|)$ is a rearrangement invariant 
Banach space if, in addition to the properties above, (7.8) holds for all measurable $\zeta$, and
\begin{equation}
\lim_{k\to\infty}q(\zeta_k)=q(\zeta)
\end{equation}
for every increasing sequence $\zeta_k\ge 0$ converging to $\zeta$. We start with the latter.
On the one hand, since $q$ is monotone, the limit in (7.9) exists, and
is $\le q(\zeta)$. On the other, the monotone convergence theorem implies that with any $f\in\cF$
\[
\int |f\zeta|=\lim_{k\to\infty}\int|f\zeta_k|\le\lim_{k\to\infty}q(\zeta_k).
\]
Taking the sup over all $f\in\cF$ we obtain $q(\zeta)\le\lim_k q(\zeta_k)$, which proves (7.9). That (7.8) holds for all 
measurable $\zeta$ now follows because $|\zeta|$ is the limit of an increasing sequence of functions in $L^\infty(X)$.

It remains to verify that $p$ and $\|\ \ \|$ are equivalent on $C^\infty(X)$. Clearly $p\le\|\ \ \|$. 
By  Lemma 7.5 
$$
\|\xi\|=\sup_{f\in\cF}\int |f\xi|\le 4\sup_{f\in\cF}\Big|\int f\xi\Big|+3\sup_{f\in\cF}\fint|f|\int|\xi|,\qquad \xi\in C^\infty(X).
$$
Equivalence follows, because the first supremum on the right is $p(\xi)$ and the last term is $\le Cp(\xi)$ 
by Lemma 4.3 and Theorem 1.3.

\end{proof}

\section{Proof of Theorem 1.5}  

The construction of a smooth, $\Ham(\omega)$ invariant function $p:C^\infty(X)\to \bR$ that is not invariant under volume preserving diffeomorphisms is based on symplectic rigidity; but linear rigidity, the easy kind, suffices. Let $V$ be a $2n\ge 4$ dimensional 
sympletic vector space over $\bR$, and $\frak{Q}$ the vector space of quadratic forms 
$Q: V\to\bR$. Linear maps of $V$ act on $\fQ$ by composition.
It is easy to construct a smooth function $t:\fQ\to\bR$ that is invariant under the symplectic group $\text{Sp}(V)$, but not under $\text{SL}(V)$. For Poisson bracket $\{\ ,\ \}$ turns $\fQ$ into a Lie algebra, and induces the adjoint action ad$_Q : \fQ\to \fQ$, 
\[
\text{ad}_Q(R)=\{Q,R\}=(\text{sgrad}\,Q)R, \qquad Q,R\in\fQ.
\]
We let $t(Q)=\text{tr}\,\text{ad}^2_Q$. Thus $t$ is a polynomial on $\fQ$. If $V\to V'$ is an isomorphism of symplectic 
vector spaces under which quadratic forms $Q,Q'$ correspond, then $t(Q)=t(Q')$.

For example, suppose that $V$ is $\bR^{2n}$ with coordinates $x_\nu, y_\nu$ and symplectic form 
$\sum_1^n dx_\nu\wedge dy_\nu$. 
Consider
$$
Q(x,y)=\sum q_{\nu} x_\nu y_\nu, \qquad q_\nu\in\bR.
$$
As sgrad$\,Q=\sum_\nu q_\nu (x_\nu\partial_{x_\nu}-y_\nu\partial_{y_\nu})$, monomials 
$x_\lambda x_\mu$, $x_\lambda y_\mu$, and $y_\lambda y_\mu$ form an eigenbasis of $\text{ad}_Q$, with eigenvalues 
$q_\lambda+q_\mu$, resp. $q_\lambda-q_\mu$, resp. $-q_\lambda-q_\mu$. Hence
\begin{equation}\label{eq8.1}\begin{split}
t(Q)&=\sum_{\lambda\le \mu}(q_\lambda+q_\mu)^2+\sum_{\lambda, \mu}(q_\lambda-q_\mu)^2+
\sum_{\lambda\ge \mu}(q_\lambda+q_\mu)^2\\
&=\sum_{\lambda= \mu}(2q_\lambda)^2+\sum_{\lambda, \mu}\big((q_\lambda+q_\mu)^2+(q_\lambda-q_\mu)^2\big)\\
&=4\sum_\lambda q_\lambda^2+2\sum_{\lambda, \mu}(q_\lambda^2+q_\mu^2)=(4n+4)\sum_\lambda q^2_\lambda.
\end{split}\end{equation}
Note that $Q$ and $R=\sum r_\nu x_\nu y_\nu$ are on the same SL$(V)$ orbit whenever $\prod q_\nu=\prod r_\nu$. We conclude $t$ is not SL$(V)$ invariant.

We need to introduce one more player. If a general quadratic form $Q:V\to\bR$ is written in a symplectic basis $z_\nu$, 
$\nu=1,\ldots,2n$, 
as $Q(z)=\sum a_{\lambda\nu}z_\lambda z_\nu$, with $a_{\lambda\nu}=a_{\nu\lambda}$, we let
$$
\Det Q=\det(a_{\lambda\nu}).
$$
Thus $\Det Q$ is independent of the choice of basis, and is even SL$(V)$ invariant.

Fix a smooth function $\varphi\colon \bR\to\bR$ such that $\varphi(s)=0$ for $|s|\le 1/2$ and $\varphi(s)=s$ for $|s|\ge 1$. 
If $\xi\in C^\infty(X)$ 
and $x$ is a critical point of $\xi$, let $Q_x=Q_{\xi,x}$ stand for the quadratic Taylor polynomial of $\xi-\xi(x)$ at $x$, a 
quadratic form on the symplectic vector space $T_xX$ (the Hessian). Given $\var>0$, critical points  $x$ of $\xi$ for which 
$|\Det Q_x|\ge\var$
form a discrete and compact, hence finite set. In particular $\xi$ has countably many nondegenerate critical points, that we denote 
$x_i$. Define
$p\colon C^\infty(X)\to\bR$ by letting
\begin{equation}
p(\xi)=\sum_i\varphi(\Det Q_{x_i})t(Q_{x_i});
\end{equation}
 we are summing over all nondegenerate critical points $x_i$ of $\xi$, or only over those for which $|\Det Q_{x_i}|>1/2$. We claim that $p$ is smooth.
 
Indeed, given $\eta\in C^\infty(X)$, let $C$ consist of its critical points $y$ for which $|\Det Q_y|\le 1/4$, a compact subset 
of $X$, and let $y_i$,
$1\le i\le k$ denote the rest of its critical points. It is possible that $k=0$, and
even that $\eta$ has no nondegenerate critical point at all.
About each $y_i$ construct a neighborhood $U_i$ so that the only critical point within $\oU_i$ is $y_i$. About each 
$y\in C$ 
construct a neighborhood $V\subset X$ with local coordinates $z_1,\ldots,z_{2n}$ so that
$\omega|V=\sum_\nu dz_\nu\wedge dz_{n+\nu}$. Let $U\subset\subset V$ be a neighborhood of $y$ consisting of 
$x$ such that
the quadratic form $Q(z)=\sum\partial_\lambda\partial_\nu\eta(x)z_\lambda z_\nu$ has determinant $|\Det Q|<1/3$. Choose a finite
 cover $\{U_{k+1},\ldots,U_l\}$ of $C$ by such neighborhoods $U$.
If $\xi\in C^\infty(X)$ is in a sufficiently small neighborhood of $\eta$,
\begin{itemize}
\item[] \ \ in each $\oU_j$, $j\le k$, $\xi$ has a single critical point, which depends smoothly on $\xi$;
\item[] \ \ all critical points $x$ of $\xi$ in $\bigcup_{j>k}\oU_j$ satisfy $|\Det Q_{\xi,x}|<1/2$; and
\item[] \ \ $\xi$ has no critical points outside $\bigcup_1^l \oU_j$.
\end{itemize}
Therefore $p$ in (8.2) is a smooth function in this neighborhood of $\eta$, hence everywhere. 

Invariance of $\Det$ and $t$ implies that $p$ is $\Ham(\omega)$ invariant. It is, however, not invariant under general 
volume preserving diffeomorphisms for the following reason. Fix a coordinate system $x_\nu, y_\nu$ on an open 
$W\subset X$, centered at some $o\in W$, such that $\omega|W=\sum dx_\nu \wedge dy_\nu$. Let $\xi\in C^\infty(X)$ be given 
by $\xi=2\sum x_\nu y_\nu$ on $W$.

The local flow of a vector field $v=\sum a_\nu(x,y)\partial_{x_\nu}+b_\nu(x,y)\partial_{y_\nu}$ preserves $\omega^n$ if and only if 
$\text{div}\, v=0$; that is, if the $(2n-1)$--form
\[
\alpha=\sum_\nu(a_\nu dx_\nu-b_\nu dy_\nu)\wedge\bigwedge\limits_{\lambda\ne \nu} dx_\lambda\wedge dy_\lambda
\]
is closed, or if locally $\alpha=d\beta$. This shows that the germ of any volume preserving flow at $o$ can be continued to a volume preserving flow on all of $X$, that will be supported in our coordinate neighborhood. With 
$c_\nu\in\bR$ consider the germ of a diffeomorphism at $o$
\begin{equation}  
(x,y)\mapsto (e^{c_\nu} x_\nu, e^{c_\nu}y_\nu)_{1\le \nu\le n}.
\end{equation}
This is the time 1 map of a volume preserving flow if $\sum c_\nu=0$. If so, there is a volume preserving diffeomorphism $g:X\to X$, 
supported in $W$, 
whose germ at $o$ is (8.3). Now $\xi$ and $\eta=\xi\circ g$ have the same critical points, and even their germs agree at all
critical points except possibly at $o$.  Hence the contributions to $p(\xi)$ and $p(\eta)$ of critical points different from $o$ are the
same. At $o$ 
$$
Q_{\xi,o}=\xi=2\sum x_\nu y_\nu, \qquad Q_{\eta,o}=\eta=2\sum e^{2c_\nu}x_\nu y_\nu.
$$
This means that $\Det Q_{\xi,o}=\Det Q_{\eta,o}=\pm 1$, while in general, in view of (8.1)
$$
t(Q_{\xi,o})=4(4n+4)n \neq 4(4n+4)\sum e^{4c_\nu}=t(Q_{\eta,o}).
$$
Therefore $p(\xi)\neq p(\eta)$, as claimed.

Note also that $p$ is discontinuous in the sup norm topology, since arbitrarily $\|\ \ \|_\infty$--close to 
$0\in  C^\infty(X)$ there are $\xi$ with a unique nondegenerate critical point $x$, where the Hessian $Q_{\xi,x}$ can
be arbitrarily prescribed.

\end{document}